\documentclass[12pt, reqno]{amsart}
\usepackage{amsmath,verbatim,amsthm,amssymb,accents,tensor,slashed}
\usepackage{tikz-cd}
\usepackage[all]{xy}
\usepackage{comment}
\newtheorem{theorem}{Theorem}[section]
\newtheorem{proposition}[theorem]{Proposition}

\newtheorem{corollary}[theorem]{Corollary} 
 
\newtheorem{lemma}[theorem]{Lemma} 

\theoremstyle{definition} 
\numberwithin{equation}{section}
\renewcommand\({\Big(} 
\renewcommand\){\Big)} 
\newcommand\<{\langle} 
\renewcommand\>{\rangle}

\newcommand\er{\eqref}
\newcommand\be[1]{\begin{equation}\label{#1}}
\newcommand\ee{\end{equation}} 
\newcommand\bm{\begin{pmatrix}}
\renewcommand\em{\end{pmatrix}}
\newcommand\bd{\begin{vmatrix}}
\newcommand\ed{\end{vmatrix}}

\newcommand\F[1]{{\sqrt{#1}}}

\newcommand\I{\int\limits}
\renewcommand\S{\sum\limits}
\renewcommand\P{\prod\limits}
\newcommand\U{\bigcup\limits}
\newcommand\C{\bigcap\limits}

\newcommand\dl{\partial}

\newcommand\op{\oplus}

\newcommand\oc{\circ}

\newcommand\iu{\cup}
\newcommand\ui{\cap}
\newcommand\ic{\subset}

\newcommand\xx{\times}

\newcommand\ap{\approx}
\newcommand\qi{\mapsto}

\renewcommand\le{\leqslant}
\renewcommand\ge{\geqslant}
\newcommand\OO{\emptyset}
\newcommand\sm{\setminus}

\renewcommand\#{\sharp}
\renewcommand\.{\cdot}
\renewcommand\:{\cdots}
\renewcommand\;{\ldots}
\newcommand\la{\alpha}
\newcommand\lb{\beta}

\newcommand\lc{\chi}
\renewcommand\lg{\gamma}
\newcommand\lG{\Gamma}

\newcommand\lD{\Delta}
\newcommand\Le{\varepsilon}

\newcommand\lk{\kappa}
\renewcommand\ll{\lambda}
\newcommand\lL{\Lambda}
\newcommand\lm{\mu}
\renewcommand{\ln}{\nu}

\newcommand\lf{\phi}

\renewcommand\lq{\psi}

\newcommand\lp{\pi}

\newcommand\lr{\varrho}
\newcommand\ls{\sigma}
\newcommand\lS{\Sigma}

\newcommand\lT{\Theta}
\newcommand\Lt{\tau}
\newcommand\lo{\omega}
\newcommand\lO{\Omega}
\newcommand\lx{\xi}

\newcommand\lz{\zeta}

\newcommand\EL{\mathcal E}

\newcommand\KL{\mathcal K} 
\newcommand\LL{\mathcal L} 
\newcommand\ML{\mathcal M}

\newcommand\PL{\mathcal P}

\newcommand\TL{\mathcal T} 
 
\newcommand\VL{\mathcal V}

\newcommand\Cl{{\mathbf C}}

\newcommand\Nl{{\mathbf N}}
 
\newcommand\Pl{{\mathbf P}}

\newcommand\Tl{{\mathbf T}}

\newcommand\q{\quad} 
\renewcommand\l{\ell}
\newcommand\m{{\boldsymbol m}} 
\newcommand\n{{\boldsymbol n}}

\renewcommand\c[1]{{\check{#1}}}

\newcommand\f[2]{{\frac{#1}{#2}}}
 
\newcommand\h[1]{{\widehat{#1}}} 
\renewcommand\o[1]{{\overline{#1}}}

\renewcommand\t[1]{{\widetilde{#1}}} 

\newcommand\w[1]{{\widetilde{#1}}} 
\newcommand\x[1]{{\mathrm{#1}}} 
\newcommand\y[1]{{\undertilde{#1}}}

\renewcommand\b{{\mathrel{\scalebox{0.8}{$\Box$}}}}

\newcommand\J[1]{\text{\reflectbox{${#1}$}}}
\newcommand\D[1]{{\slashed{#1}}}

\renewcommand\ni{\J{\in}}

\DeclareFontFamily{U}{mathb}{\hyphenchar\font45}
\DeclareFontShape{U}{mathb}{m}{n}{
 <5><6><7><8><9><10> gen * mathb
 <10.95> mathb10
}{}
\DeclareSymbolFont{mathb}{U}{mathb}{m}{n}
\DeclareMathSymbol{\VDash}3{mathb}{"28}
\begin{document}
\setcounter{section}{-1}

\title{Stratified Hilbert Modules on Bounded Symmetric Domains}
\author{Harald Upmeier} 
\medskip

\address{Fachbereich Mathematik, Universit\"at Marburg, D-35032 Marburg, Germany}
\email{upmeier{@}mathematik.uni-marburg.de}

\dedicatory{Dedicated to the Memory of J\"org Eschmeier}



\begin{abstract} We analyze the "eigenbundle" (localization bundle) of certain Hilbert modules over bounded symmetric domains of rank $r,$ giving rise to complex-analytic fibre spaces which are stratified of length $r+1.$ The fibres are described in terms of K\"ahler geometry as line bundle sections over flag manifolds, and the metric embedding is determined by taking derivatives of reproducing kernel functions. Important examples are the determinantal ideals defined by vanishing conditions along the various strata of the stratification. 
\end{abstract}

\maketitle

\section{Introduction}
J\"org Eschmeier was a master in the interplay between operator theory and multi-variable complex analysis, in particular in its modern sheaf-theoretic form. An interesting concept in this respect is the "eigenbundle" of a Hilbert module arising from a commuting tuple of non-selfadjoint operators $T_1,\;,T_d.$ Just as the spectral theorem is the basic tool for the analysis of self-adjoint operators, the eigenbundle plays a similar role in the non-selfadjoint case, naturally involving multi-variable complex analysis instead of "real" measure theory. 

In the original approach by Cowen-Douglas \cite{CD} for Hilbert modules $H$ of holomorphic functions on a bounded domain 
$D\ic\Cl^d,$ the eigenbundle, denoted by $H_D,$ is a genuine holomorphic vector bundle on $D$ whose hermitian geometry determines the underlying operator tuple. In more general situations \cite{DMV,DG} the eigenbundle will have singularities along certain subvarieties of $D,$ making the geometric aspects more complicated. For example, if $H=\o I$ is the Hilbert closure of a prime ideal $I$ of polynomials, whose vanishing locus $X$ is smooth, then by a result of Duan-Guo \cite{DG} the eigenbundle 
$H_D$ has rank 1 on the regular set $D\sm X,$ whereas $H_X:=H_D|_X$ is isomorphic to the (dual) normal bundle of $X.$ Thus we have a "stratification" of length 2. In this paper, we study $K$-invariant Hilbert modules $H$ over bounded symmetric domains $D=G/K$ of rank $r,$ leading to singular vector bundles which are stratified of length $r+1.$ This study was initiated in \cite{U3} where the eigenbundle of certain polynomial ideals $J^\ll,$ for a given partition $\ll$ of length $\le r,$ was determined explicitly. These ideals are not prime ideals except for the "fundamental" partitions $\ll=(1,\;,1,0,\;,0).$ The main result of \cite{U3} describes the fibres of the eigenbundle using representation theory of the compact Lie group $K.$

The current paper extends and generalizes this analysis. For the partition ideals $J^\ll$ and Hilbert closures $H=\o J^\ll$ we construct a Hilbert space embedding of the eigenbundle by taking certain derivatives of the reproducing kernel function of $H.$ This is important to study the hermitian structure and is related to the "jet construction" introduced in \cite{DMV}. Moreover, we give a holomorphic characterization of the eigenbundle in terms of holomorphic sections of line bundles over a flag manifold. Such a geometric characterization may also hold in more general situations. Beyond the setting of the partition ideals we consider arbitrary $K$-invariant polynomial ideals, in particular the so-called "determinantal" ideals which have a direct geometric meaning. 
	
\section{Hilbert Modules and their Eigenbundle}	
Let $D$ be a bounded domain in a finite dimensional complex vector space $E\ap\Cl^d.$ Denote by $\PL_E\ap\Cl[z_1,\;,z_d]$ the algebra of all polynomials on $E.$ A Hilbert space $H$ of holomorphic functions $f$ on $D$ (supposed to be scalar-valued) is called a {\bf Hilbert module} if for any polynomial $p\in\PL_E$ the multiplication operator $T_pf:=pf$ leaves $H$ invariant and is bounded. Using the adjoint operators $T_p^*,$ the closed linear subspace
\be{1}H_\lz:=\{f\in H:\ T_p^*f=\o{p(\lz)}f\ \forall\ p\in\PL_E\}\ee
is called the {\bf joint eigenspace} at $\lz\in D.$ Since $T_pT_q=T_{pq}$ for polynomials $p,q$ it suffices to consider linear functionals or just the coordinate functions. The disjoint union
$$H_D=\U_{\lz\in D}H_\lz$$
becomes a subbundle of the trivial vector bundle $D\xx H,$ which is called the {\bf eigenbundle} of $H,$ although it is not locally trivial in general. One also requires that the fibres have finite dimension and their union is total in $H.$ The map $\lf\qi[\lf]$ from $H_\lz$ to $H/\o{\ML_\lz H}$ is a Hilbert space isomorphism, with inverse map 
$$H/\o{\ML_\lz H}\to H_\lz,\q f+\o{\ML_\lz H}\qi\lp_\lz f,$$ 
where $\lp_\lz: H\to H_\lz$ is the orthogonal projection. Thus $H_\lz$ becomes the "quotient module" for the submodule 
$\o{\ML_\lz H}.$ 

Classical examples of Hilbert modules are the {\bf Bergman space} $H^2(D)$ of square-integrable holomorphic functions, whose reproducing kernel is called the Bergman kernel, and the {\bf Hardy space} $H^2(\dl D)$ if $D$ has a smooth boundary $\dl D.$ For general Hilbert modules $H$, a {\bf reproducing kernel function} is a sesqui-holomorphic function $\KL(z,\lz)$ on $D\xx D$ such that for each $\lz\in D$ the holomorphic function 
$$\KL_\lz(z):=\KL(z,\lz)$$ 
belongs to $H,$ and we have
$$\lq(z)=(\KL_z|\lq)_ H$$ 
for all $\lq\in  H$ and $z\in D.$ Here $(\lf|\lq)_H$ is the inner product, anti-linear in $\lf.$ Thus $H$ is the closed linear span of the holomorphic functions $\KL_\lz,$ where 
$\lz\in D$ is arbitrary. If $\lf_\la$ is any orthonormal basis of $H$ then
$$\KL(z,\lz)=\S_\la\ \lf_\la(z)\o{\lf_\la(\lz)}.$$ 
For each $\lz\in D$ we have $\KL_\lz\in H_\lz,$ as follows from the identity 
$$(T_p^*\KL_\lz|\lq)_ H=(\KL_\lz|p\lq)_ H=p(\lz)\lq(\lz)=p(\lz)(\KL_\lz|\lq)_ H=(\o{p(\lz)}\KL_\lz|\lq)_ H$$
for $p\in\PL_E$ and $\lq\in  H.$ 

If the reproducing kernel $\KL$ has no zeros (e.g., the Bergman kernel of a strongly pseudo-convex domain) then the eigenbundle 
$H_D$ is spanned by the functions $\KL_\lz$ and hence becomes a {\bf hermitian holomorphic line bundle}. In more general cases the kernel function vanishes along certain analytic subvarieties of $D$ and the eigenbundle is not locally trivial, its fibre dimension can jump along the varieties and we obtain a {\bf singular vector bundle} on $D,$ also called a "linearly fibered complex analytic space" \cite{F}. Such singular vector bundles are important in Several Complex Variables since they are in duality with the category of {\bf coherent analytic module sheaves}, whereas (regular) vector bundles correspond to locally free sheaves. In \cite{BMP} the connection to coherent analytic module sheaves associated with $H$ is made explicit.

An important class of Hilbert modules is given by the Hilbert closure $H=\o I$ of a polynomial ideal $I\ic\PL_E.$ In this case the fibres \er{1} of the eigenbundle have finite dimension. More precisely, by \cite{DG} we have $H_D\ap I_E|_D$ for the 
"localization bundle" $I_E$ over $E,$ with fibre
$$I_\lz:=I/\ML_\lz I$$
at $\lz\in E.$ For any set of generators $p_1,\;,p_t$ of $I$ the linear map
$$\Cl^t\to I_\lz,\ (a_1,\;,a_t)\qi\ML_\lz I+\S_{i=1}^t a_i\ p_i$$
is surjective, showing that $\dim I_\lz\le t.$ In \cite[Lemma 2.3]{BMP} it is shown that $f\in H_\lz$ satisfies
$$\o{p_i(\lz)}f=(p_i|f)_ H\ \KL_\lz$$
for all $i.$ This implies that the eigenbundle $H_D$ restricted to the open dense subset 
$$\c D:=\U_{j=1}^t\{\lz\in D:\ p_j(\lz)\ne 0\}$$
of $D$ is a holomorphic line bundle spanned by the reproducing kernel $\KL_\lz,\ \lz\in\c D.$

The behavior of $H_D$ on the singular set $D\sm\c D$ is more complicated and has so far been studied mostly when the vanishing locus of the reproducing kernel is a smooth subvariety of $D,$ for example given as a complete intersection of a regular sequence of polynomials. The case where $I$ is a {\bf prime ideal} whose vanishing locus $X$ consists of {\bf smooth points} has been studied by Duan-Guo \cite{DG}. They showed that for $\lz\in D\sm X$
$$H_\lz=\<\KL_\lz\>$$
is the 1-dimensional span of the reproducing kernel vector, whereas for $\lz\in X$
$$H_\lz\approx T_\lz^\perp(X)$$
is isomorphic to the normal space (more precisely its linear dual.) Thus we have a {\bf stratification of length 2}. We consider a more complicated situation for bounded symmetric domains $D$ of arbitrary rank $r,$ where we have a {\bf stratification of length $r+1$}, the relevant varieties are not smooth and the ideal $I$ is not prime in general.

\section{$K$-invariant Ideals on Bounded Symmetric Domains}
Let $D=G/K$ be an irreducible bounded symmetric domain of rank $r$, realized as the (spectral) unit ball of a hermitian Jordan triple ($J^*$-triple) $E.$ Let $\{uv^*w\}\in E$ denote the Jordan triple product of $u,v,w\in E.$ The compact Lie group $K$ acts by linear transformations on $E$ preserving the Jordan triple product. For background on the Jordan theoretic approach towards symmetric domains, see \cite{Ar,C,FK2,L2,U1}. Let $\PL_E$ denote the algebra of all polynomials on $E.$ Under the natural action 
$$(k\.f)(z):=f(k^{-1}z)$$
of $k\in K$ on functions $f$ on $E$ the polynomial algebra $\PL_E$ has a {\bf Peter-Weyl decomposition} \cite{S,FK1}
$$\PL_E=\S_{\ll\in\Nl_+^r}\PL_E^\ll$$
into pairwise inequivalent irreducible $K$-modules $\PL_E^\ll.$  Here $\Nl_+^r$ denotes the set of all {\bf partitions}
$$\ll=(\ll_1,\;,\ll_r)$$
of integers $\ll_1\ge\;\ge\ll_r\ge 0.$ The polynomials in $\PL_E^\ll$ are homogeneous of degree $|\ll|:=\ll_1+\;+\ll_r.$ We often identify a partition $\ll$ with its {\bf Young diagram} 
$$[\ll]=\{(i,j):\ 1\le i\le r,\ 1\le j\le\ll_i\}.$$
For fixed $\lm\in\Nl_+^r$ denote by
$$\lp^\lm:\PL_E\to\PL_E^\lm,\q f\qi\lp^\lm f=:f^\lm$$
the $K$-invariant projection onto $\PL_E^\lm.$ As a consequence of Schur orthogonality we have \cite[]{U3}
\be{2}f^\lm(z)=\I_K dk\ \lc_\lm(k)\ f(kz),\ee
where $\lc_\lm$ denotes the character of the $K$-representation on $\PL_E^\lm.$ 

An ideal $I\ic\PL_E$ is called $K$-invariant if $k\.f\in I$ for all $k\in K$ and $f\in I.$ A similar definition applies to Hilbert modules of holomorphic funtions on a $K$-invariant domain. The formula \er{2} implies that a $K$-invariant ideal (resp., Hilbert module) is a direct sum (resp., Hilbert sum) of its Peter-Weyl subspaces $\PL_E^\lm.$ 

For a given partition $\ll$ denote by $J^\ll\ic\PL_E$ the $K$-invariant ideal generated by $\PL_E^\ll.$ The first main result of \cite{U3} asserts that $J^\ll$ has the Peter-Weyl decomposition
\be{3}J^\ll=\S_{\lm\ge\ll}\PL_E^\lm,\ee
where $\ll\le\lm$ means $\ll_i\le\lm_i$ for all $i.$ This is equivalent to the inclusion $[\ll]\ic[\lm]$ for the corresponding Young diagrams. As a consequence, $J^\lm\ic J^\ll$ if and only if $\lm\ge\ll.$ In this section we show that these "partition" ideals $J^\ll$ are fundamental for the study of general $K$-invariant ideals. Given a $K$-invariant ideal $I,$ define 
$I^\#:=\{\ll\in\Nl_+^r:\ J^\ll\ic I\}.$ 

\begin{proposition}\label{a} Let $I\ic\PL_E$ be a $K$-invariant ideal. Then there is a finite set $\lL\ic I^\#$ of partitions such that
$$I=\S_{\ll\in\lL}J^\ll.$$
\end{proposition}
\begin{proof} If $f\in I$ then \er{2} shows that $f^\lm$ also belongs to $I$ for all $\lm\in\Nl_+^r.$ Let 
$f_1,\;,f_t$ be a finite set of generators of $I.$ Then their non-zero $K$-homogeneous parts belong to $I$ and form a finite set of generators. Thus we may assume that each $f_s\in\PL_E^{\ll_s}$ for some partition $\ll_s.$ We claim that
$$I=\S_{s=1}^t J^{\ll_s}=:J.$$
Since $f_s\in I\ui\PL_E^{\ll_s}$ is non-zero, $I$ is $K$-invariant and $\PL_E^{\ll_s}$ is irreducible, it follows that 
$\PL_E^{\ll_s}\ic I$ and hence $J^{\ll_s}\ic I.$ Thus $J\ic I.$ Conversely, each generator $f_s$ of $I$ belongs to $J.$ Therefore $I\ic J.$
\end{proof}

A subset $\lL\ic I^\#$ is called "full" if
$$I=\S_{\ll\in\lL}J^\ll.$$ 
A subset $A\ic I^\#$ is called "minimal" if $\la\in A$ and $\ll<\la$ implies $\ll\notin I^\#,$ i.e., $J^\ll\D{\ic}I.$ 

\begin{lemma}\label{b} Let $A\ic I^\#$ be minimal and $\lL\ic I^\#$ be full. Then $A\ic\lL.$
\end{lemma}
\begin{proof} Suppose there exists $\la\in A$ such that $\la\notin\lL.$ Let $p\in\PL_E^\la.$ Then 
$p\in J^\la\ic I=\S_{\ll\in\lL}J^\ll$ and therefore
$$p=\S_{\ll\in\lL}f_\ll\q\mbox{(finite sum)}$$
for some $f_\ll\in J^\ll.$ By \er{3} we have $f_\ll^\la=0$ unless $\la\ge\ll.$ Since $\la\notin\lL$ this implies
$$p=p^\la=\S_{\ll\in\lL}f_\ll^\la=\S_{\ll\in\lL,\ \la\ge\ll}f_\ll^\la=\S_{\ll\in\lL,\ \la>\ll}f_\ll^\la=0$$
since $\{\ll\in\lL:\ \ll<\la\}=\OO.$ This is a contradiction.
\end{proof}

\begin{corollary} Every minimal set $A\ic I^\#$ is finite.
\end{corollary}
\begin{proof} By Proposition \ref{a} there exists a full set $\lL\ic I^\#$ which is finite. By Lemma \ref{b} we have $A\ic\lL.$ Hence $A$ is finite.
\end{proof}

\begin{proposition} There exists a (finite) set $A\ic I^\#$ which is both full and minimal.
\end{proposition}
\begin{proof} For any finite subset $\lL\ic I^\#$ we put
$$|\lL|:=\S_{\ll\in\lL}|\ll|,$$
where $|\ll|:=\ll_1+\;+\ll_r.$ By Proposition \ref{a} there exists a finite subset $\lL\ic I^\#$ which is full. Put
$$k:=\min\{|\lL|:\ \lL\ic I^\#\ \mbox{full and finite}\}.$$
Then there exists a full and finite set $A\ic I^\#$ such that $|A|=k.$ We claim that $A$ is minimal. Suppose there exist 
$\la\in A$ and $\ll\in I^\#$ with $\ll<\la.$ Then $J^\la\ic J^\ll\ic I$ which shows that the finite set 
$\lL=(A\sm\{\la\})\iu\{\ll\}$ is still full. On the other hand, we have $|\ll|<|\la|$ and hence
$$|\lL|=|A\sm\{\la\}|+|\ll|=|A|-|\la|+|\ll|<|A|=k.$$
This contradiction shows that $A$ is minimal.
\end{proof}

The arguments above show that there is a unique finite set $I^\#_{min}\ic I^\#$ which is both full and minimal. We formulate this as

\begin{proposition}\label{c} Let $I\ic\PL_E$ be a $K$-invariant ideal. There exists a unique finite set $\lL\ic\Nl_+^r$ such that
\be{4}I=\S_{\ll\in\lL}J^\ll\ee
and $J^\lm\D{\ic}I$ if $\lm<\ll$ for some $\ll\in\lL.$
\end{proposition} 

For example, the $n$-th power of the maximal ideal $\ML_0$ has the form
$$\ML_0^n=\S_{|\ll|=n}J^\ll$$
since the polynomials in $\PL_E^\ll$ are homogeneous of degree $|\ll|.$ For $n=1$ we have $\ML_0=I^{1,0,\;,0}$ since 
$\PL_E^{1,0,\;,0}=E^*$ is the linear dual space of $E.$ More interesting examples will be studied in the next section. 

As an application of Proposition \ref{c} we determine for each $K$-invariant ideal $I$ the "maximal fibre" $I_0=H_0$ of the eigenbundle. 

\begin{theorem}\label{d} Let $I$ be a $K$-invariant ideal, written in the "minimal" form \er{4}. Then any Hilbert module closure $H=\o I$ has the maximal fibre
$$H_0=\S_{\ll\in\lL}\PL_E^\ll\q\mbox{(direct sum)}$$
at the origin.
\end{theorem}
\begin{proof} We first show that $\PL_E^\ll\ic H_0$ for $\ll\in\lL.$ Let $p\in\PL_E^\ll$ and $\l\in E^*$ a linear form on $E.$ By \cite{U2} we have
$$T_\l^*p=\S_{j=1}^r(T_\l^*p)^{\ll-\Le_j}.$$
If $T_\l^*p\ne 0$ then $(T_\l^*p|q)\ne 0$ for some $q\in J^\lm$ with $\lm\in\lL.$ By \er{3} we have
$$q=\S_{\ln\ge\lm}q^\ln.$$ 
Therefore
$$(T_\l^*p|q)=\S_{j=1}^r((T_\l^*p)^{\ll-\Le_j}|q)=\S_{j=1}^r\S_{\ln\ge\lm}((T_\l^*p)^{\ll-\Le_j}|q^\ln).$$
It follows that there exists $j$ such that $\ll-\Le_j=\ln.$ Therefore $\ll>\ln\ge\lm.$ Since both $\ll,\lm\in\lL$ this contradicts the fact that $\lL$ is minimal.

Conversely, suppose there exists $\lf\in H_0$ which is orthogonal to $\S_{\ll\in\lL}\PL_E^\ll.$ By averaging over $K$ we may assume that $\lf\in\PL_E^\lm$ for some $\lm\notin\lL.$ We can write $\lf\in H=\o I$ as
$$\lf=\S_{\lL\ni\ll<\lm}\lf_\ll$$   
where $\lf_\ll$ belongs to the ideal $J^\ll.$ Since $J^\ll$ is generated by $\PL_E^\ll,$ there exist 
$g_\ll^i\in\ML_0,\ a_\ll^i\in\Cl$ and $p_i^\ll\in\PL_E^\ll$ such that
$$\lf=\S_{\lL\ni\ll<\lm}\S_i(g_\ll^i+a_\ll^i)\ p_i^\ll=\S_{\lL\ni\ll<\lm}\S_i g_\ll^i\ p_i^\ll
+\S_{\lL\ni\ll<\lm}\S_i a_\ll^i\ p_i^\ll.$$
Since $\lf\in\PL_E^\lm$ applying the projection $\lp^\lm$ yields 
$$\lf=\lf^\lm=\S_{\lL\ni\ll<\lm}\S_i (g_\ll^i\ p_i^\ll)^\lm.$$
It follows that
$$(\lf|\lf)=\S_{\lL\ni\ll<\lm}\S_i(\lf|(g_\ll^i\ p_i^\ll)^\lm)=\S_{\lL\ni\ll<\lm}\S_i(\lf|g_\ll^i\ p_i^\ll)
=\S_{\lL\ni\ll<\lm}\S_i(T_{g_\ll^i}^*\lf|p_i^\ll)=0.$$
Therefore $\lf=0.$
\end{proof}

As a special case of Theorem \ref{d} we have
\be{6}J_0^\ll=\PL_E^\ll,\ee
already proved in \cite{U3}
 
The description of the eigenbundle at non-zero points $\lz$ is more complicated and depends on the {\bf rank} of $\lz$ (in a Jordan theoretic sense). For $0\le\l\le r$ define the {\bf Kepler manifold} 
$$\c E_\l:=\{\lz\in E:\ \x{rank}(\lz)=\l\}.$$
The complexification $\h K$ of $K$ is a complex subgroup of $\x{GL}(E)$ called the "structure group." It acts transitively on each $\c E_\l.$ An element $c\in E$ satisfying $\{cc^*c\}=2c$ is called a {\bf tripotent} (triple idempotent). Let 
$S_\l\ic\c E_\l$ denote the compact $K$-homogeneous manifold of all tripotents of rank $\l.$ It is shown in \cite{U3} that the eigenbundle $H_D$ restricted to each "stratum" $\c E_\l$ is a {\bf homogeneous holomorphic vector bundle} under the $\h K$-action, induced by the fibre $H_c$ at any tripotent $c$ of rank $\l.$ Hence it suffices to study the fibre at a tripotent 
$c.$ In terms of the Peirce decomposition \cite{L1,L2} 
$$E=E_c^2\op E_c^1\op E_c^0,$$
where $E_c^j:=\{z\in E:\ \{cc^*z\}=jz\},$ the tangent space at $c$ is given by
$$T_c(\c E_\l)=E_c^2\op E_c^1.$$
Therefore the Peirce 0-space $E_c^0$ can be identified with the {\bf normal space} at $c.$ We abbreviate 
$E_c:=E_c^2,\ E^c:=E_c^0$ and note that $E_c$ and $E^c$ are irreducible $J^*$-subtriples of $E$ having rank $\l$ and $r-\l,$ respectively. (The Peirce 1-space $E_c^1$ is also a Jordan subtriple, but not necessarily irreducible). 

The principal tool to analyze the fibre $I_c\ap H_c$ is the {\bf normal projection map}
\be{5}\lp_c:\PL_E\to\PL_W,\ \lp_c f(w):=f(c+w)\ee
onto the polynomial algebra $\PL_W$ of the Peirce 0-space $W:=E^c$ of a tripotent $c$ of rank $\l.$ Since $W$ is an irreducible 
$J^*$-triple of complementary rank $r-\l,$ the polynomial algebra $\PL_W$ has its own Peter-Weyl decomposition
$$\PL_W=\S_{\la\in\Nl_+^{r-\l}}\PL_W^\la$$
with respect to the automorphism group $K_W.$ Here we write partitions $\la\in\Nl_+^{r-\l}$ in the form
$\la=(\la_{\l+1},\;,\la_r)$ with $\la_{\l+1}\ge\;\ge\la_r\ge 0.$ For any partition $\ll$ of length $r$ we define the "truncated partition"
$$\ll^*:=(\ll_{\l+1},\;,\ll_r)\in\Nl_+^{r-\l}.$$
Let $J_W^{\ll^*}\ic\PL_W$ denote the ideal generated by $\PL_W^{\ll^*}.$

\begin{proposition} For any $K$-invariant ideal $I,$ written in the "minimal" form \er{4}, and any tripotent $c$ with Peirce 0-space $W,$ the normal projection map \er{5} maps $I$ into the $K_W$-invariant ideal
$$I_W:=\S_{\ll\in\lL}J_W^{\ll^*}\ic\PL_W,$$
and hence induces a mapping
\be{8}I_c=I/\ML_c I\xrightarrow{\lp_c}I_W/\ML_{W,0}I_W=\y{I_W}_0\ee
between the localization of $I$ at $c$ and the maximal fibre $\y{I_W}_0$ relative to $W.$
\end{proposition}
\begin{proof} By \cite[Theorem]{U3} the normal projection map satisfies
\be{14}\lp_c:J^\ll\to J_W^{\ll^*}\ee
for any partition $\ll\in\Nl_+^r.$ In other words, if $f\in\PL_E$ has only $K$-components for $\lm\ge\ll,$ then 
$\lp_c f\in\PL_W$ has only $K_W$-components for partitions $\la\in\Nl_+^{r-\l}$ satisfying $\la\ge\ll^*.$ One can show that each such partition $\la$ occurs. Taking the (finite) sum over $\ll\in\lL,$ the first assertion follows. Since 
$(\lp_c f)(0)=f(c+0)=f(c)$ it follows that $\lp_c$ maps the maximal ideal $\ML_c$ to the maximal ideal $\ML_{W,0}\ic\PL_W$ relative to $W.$ This implies the second assertion.
\end{proof}

The main result of \cite{U3} asserts that for $I=J^\ll$ the map
\be{7}J_c^\ll=J^\ll/\ML_c J^\ll\xrightarrow{\lp_c}J_W^{\ll^*}/\ML_{W,0}J_W^{\ll^*}=\y{J_W^{\ll^*}}_0\ee
is an isomorphism. Since \er{6} applied to $W$ yields an isomorphism 
$$\y{J_W^{\ll^*}}_0\xrightarrow[\ap]{\lp_W^{\ll^*}}\PL_W^{\ll^*}$$
by taking the lowest $K_W$-type $\lf^{\ll^*}=\lp_W^{\ll^*}\lf$ of $\lf\in J_W^{\ll^*},$ \er{7} amounts to the isomorphism
\be{19}J_c^\ll\xrightarrow[\ap]{\lp_c^{\ll^*}}\PL_W^{\ll^*}\ee
where $\lp_c^{\ll^*}f:=(\lp_c f)^{\ll^*}$ denotes the lowest $K_W$-type of $\lp_c f\in J_W^{\ll^*}.$ It is likely that in general the map \er{8} is an isomorphism. This would reduce the description of the fibres $I_c$ to the combinatorial problem of finding a mimimal subset of the set $\{\ll^*:\ \ll\in\lL\}.$ This will be carried out separately.

\section{Determinantal ideals}
While the ideals $J^\ll$ are defined in terms of representation theory and make no sense beyond the Jordan theoretic setting, we now introduce $K$-invariant ideals which are defined by vanishing conditions along certain subvarieties of $E.$ The resulting analysis could shed some light on more general situations. For any ideal $I\ic\PL_E$ define the {\bf vanishing locus}
$$\VL^I:=\{z\in E:\ p(z)=0\ \forall\ p\in I\}.$$
A closed subset $X\ic E$ of the form $X=\VL^I$ for some ideal $I\ic\PL_E$ is called an algebraic variety. Conversely, for an algebraic variety $X\ic E$ consider the {\bf vanishing ideal}
$$\ML_X:=\{p\in\PL_E:\ p|_X=0\}=\{p\in\PL_E:\ p(\lz)=0\ \forall\ \lz\in X\}=\C_{\lz\in X}\ML_\lz.$$
Here $\ML_\lz\ic\PL_E$ denotes the maximal ideal of all polynomials vanishing at $\lz\in E.$ By Hilbert's basis theorem, $\ML_X$ has a finite set of generators. Hence any algebraic variety is the vanishing locus of finitely many polynomials. An ideal $I$ is prime if and only if the algebraic variety $\VL^I$ is irreducible. Conversely, an algebraic variety $X$ is irreducible if and only if $\ML_X$ is a prime ideal. In general, we have 
$$\VL^{\ML_X}=X$$
and, by Hilbert's Nullstellensatz, 
$$\ML_{\VL^I}=\F{ I}\q\mbox{(radical)}.$$
We say that $f\in\PL_E$ has {\bf order of vanishing} $\x{ord}_\lz(f)\ge n$ at $\lz\in E$ if 
$f\in\ML_\lz^n.$ Equivalently, $\x{ord}_\lz(f)>n$ if the $n$-th Taylor polynomial of $f$ at $\lz$ vanishes. Given an irreducible algebraic variety $X\ic E$ one defines the {\bf $n$-th symbolic power}
$$\ML_X^{(n)}:=\{f\in\PL_E:\ \x{ord}_\lz(f)\ge n\ \forall\ \lz\in X\}=\C_{\lz\in X}\ML_\lz^n$$
consisting of all polynomials which vanish of order $\ge n$ on $X.$ These are "primary ideals" associated to the prime ideal 
$$\ML_X=\ML_X^{(1)}=\{f\in\PL_E:\ \x{ord}_\lz(f)>0\ \forall\ \lz\in X\}.$$
The algebraic power $\ML_X^n$ is contained in $\ML_X^{(n)}$ but is generally smaller if $n>1.$ For a thorough discussion of such matters, cf. \cite{dCEP}.

For any algebraic variety $X$ let $\c X$ denote the open dense subset of all smooth (regular) points. The complement $X\sm\c X$ is the singular set of $X.$ A nested sequence $X_0\ic X_1\ic\;\ic X_{r-1}$ of algebraic varieties is called a 
{\bf stratification} if for each $\l\le 1$ the algebraic variety $X_\l$ has the singular set $X_{\l-1},$ i.e., the smoooth points
\be{9}\c X_\l=X_\l\sm X_{\l-1}.\ee
We assume that the lowest stratum $X_0=\c X_0$ is smooth, so that $X_{-1}=\OO.$ We put $\c X_r:=E\sm X_{r-1}$ as an open dense subset of $E.$ The sets $\c X_\l$ for $0\le\l\le r$ are called the strata of the stratification. The $\l$-th stratum $\c X_\l$ has the closure 
$$X_\l=\U_{0\le h\le\l}\c X_h.$$ 
Thus $X_0$ is the only closed stratum. Now consider an $r$-tuple 
\be{11}\n:=(n_1,\;,n_r)\in\Nl_+^r\ee
of integers $n_1\ge n_2\ge\;\ge n_r\ge0,$ and define the {\bf joint symbolic power} 
\be{10}\ML_{X_0,\;,X_{r-1}}^{(\n)}:=\C_{j=1}^r\ML_{X_{j-1}}^{(n_j)}\ee
consisting of all polynomials which vanish of order $\ge n_j$ along the subvariety $X_{j-1}.$ As special cases we have
$$\ML_{X_j}^{(n)}=\ML_{X_0,\;,X_{r-1}}^{(n^{(j+1)},0^{r-j-1})}$$
and in particular, for the prime ideal (if $X_j$ is irreducible) 
$$\ML_{X_j}=\ML_{X_0,\;,X_{r-1}}^{(1^{(j+1)},0^{(r-j-1)})}.$$
An irreducible $J^*$-triple $E$ has a canonical stratification
$$\{0\}=\h E_0\ic\h E_1\ic\;\ic\h E_{r-1}\ic\h E_r=E,$$
where $\h E_j$ is the set of all elements $\lz\in E$ of rank $\le j,$ called the $j$-th {\bf Kepler variety} (in a Jordan theoretic setting). The smooth points of $\h E_j$ form the Kepler manifold $\c E_j$ defined above. Hence the singular set of 
$\h E_j$ is $\h E_{j-1}$ so that the condition \er{9} is satisfied. As a special case of \er{10} define the joint symbolic power
$$\ML_E^{(\n)}:=\ML_{\h E_0,\;,\h E_{r-1}}^{(n_1,\;,n_r)}=\C_{j=1}^r\ML_{\h E_{j-1}}^{(n_j)}$$
associated with the decreasing tuple \er{11}. Thus $\ML_E^{(\n)}$ consists of all polynomials on $E$ which vanish of order 
$\ge n_j$ along the subvariety $\h E_{j-1}.$ These ideals are called "determinantal ideals" since the Kepler varieties are defined by vanishing conditions for Jordan theoretic determinants and minors. Since the Kepler varieties $\h E_j$ are $K$-invariant, $\ML_E^{(\n)}$ is a $K$-invariant ideal. As such, it is a sum of certain 
"partition" ideals $J^\ll.$ Our next result makes this precise. 

An irreducible Jordan algebra $E$ with unit element $e$ has a unique {\bf determinant polynomial} $\lD_e:E\to\Cl$ normalized by $\lD_e(e)=1$ \cite{FK2,N}. For the matrix algebra 
$E=\Cl^{r\xx r}$ and the symmetric matrices $E=\Cl_{\x{sym}}^{r\xx r}$ this is the usual determinant. For the antisymmetric matrices $E=\Cl_{\x{asym}}^{2r\xx 2r}$ we obtain the Pfaffian determinant instead. The determinant polynomial $\lD_e$ has the semi-invariance property
\be{12}\lD_e(kz)=\lD_e(ke)\lD_e(z)\ee  
for all $k\in K$ and $z\in E.$ The map $\lc:K\to\Tl$ defined by
$$\lc(k):=\lD_e(ke)$$
is a character of $K.$ It follows that for any $k\in K$
$$\lD_{ke}(z):=\lD_e(k^{-1}z)$$
is a Jordan determinant normalized at $ke.$ 

\begin{theorem}\label{e} For each tuple \er{11} the joint symbolic power has the decomposition
$$\ML_E^{(\n)}=\C_{j=1}^r\ML_{E_{j-1}}^{(n_j)}=\S_{\ll\in\Nl_{(\n)}^r}J^\ll$$
where
$$\Nl_{(\n)}^r:=\{\ll\in\Nl_+^r:\ \ll_j+\;+\ll_r\ge n_j\ \forall\ 1\le j\le r\}.$$
\end{theorem}
\begin{proof} By highest weight theory \cite{U1} the space $\PL_E^\ll$ is spanned by polynomials
\be{23}N^\ll:=N_1^{\ll_1-\ll_2}N_2^{\ll_2-\ll_3}\:N_r^{\ll_r},\ee
where $e_1,\;,e_r$ is any frame of $E$ and
$$N_m(z):=\lD_e(P_ez)$$
denotes the Jordan theoretic minor for the tripotent $e=e_{[m]}=e_1+\;+e_m$ and its Peirce 2-space $E_e.$ Here $P_e:E\to E_e$ is the Peirce 2-projection. Let $\lz\in\h E_j$ and $m>j.$ By the spectral theorem applied to $E_e$ there exist a frame $c_1,\;,c_m$ of $E_e$ and $a_i\in\Cl$ such that
$$P_e\lz=\S_{i=1}^m a_i c_i.$$
Choose $k\in K_{E_e}$ such that $ke_i=c_i$ for all $i\le m.$ By \cite[Theorem 1]{N} we have
$$\lD_e(u+\S_{i=1}^m a_i e_i)=\S_{T\ic\{1,\;,m\}}\lD_{e-e_T}(P_{e-e_T}u)\ \P_{i\in T}a_i$$
for all $u\in E_e,$ where $e_T:=\S_{i\in T}e_i.$ For $z\in E$ it follows that
$$N_e(z+\lz)=\lD_e(P_e(z+\lz))=\lD_e(P_e z+\S_{i=1}^m a_i c_i)=\lD_e(k(k^{-1}P_e z+\S_{i=1}^m a_i e_i))$$
$$=\lD_e(ke)\ \lD_e(k^{-1}P_e z+\S_{i=1}^m a_i e_i)
=\lD_e(ke)\ \S_{T\ic\{1,\;,m\}}\lD_{e-e_T}(P_{e-e_T}k^{-1}P_ez)\ \P_{i\in T}a_i.$$
Since $P_e\lz$ has rank $\le j$ it follows that at most $j$ coefficients $a_i$ are non-zero, and hence $\P_{i\in T}a_i=0$ whenever the cardinality $|T|>j.$ Hence in the sum only subsets $T$ with $|T|\le j$ occur. Since the polynomial 
$\lD_{e-e_T}\oc P_{e-e_T}k^{-1}P_e$ on $E$ is homogeneous of degree $m-|T|=\x{rank}(E_{e-e_T})$ we obtain
$$\x{ord}_\lz(N_m)=\x{ord}_0(N_m(\lz+\.))\ge\mathop{\min}\limits_{|T|\le j}\x{ord}_0(\lD_{e-e_T}\oc P_{e-e_T}k^{-1}P_e)
=\mathop{\min}\limits_{|T|\le j}(m-|T|)=m-j.$$
It follows that
$$\x{ord}_\lz(N^\ll)=\S_{m=1}^r(\ll_m-\ll_{m+1})\ \x{ord}_\lz(N_m)\ge\S_{m=j+1}^r(\ll_m-\ll_{m+1})\ \x{ord}_\lz(N_m)$$
$$\ge\S_{m=j+1}^r(\ll_m-\ll_{m+1})(m-j)=\S_{m=j+1}^r\ll_m.$$
Hence the estimate $\x{ord}_\lz(f)\ge\S_{m=j+1}^r\ll_m$ holds for all $f\in\PL_E^\ll$ and a fortiori for all $f\in J^\ll.$ This shows that $J^\ll\ic\ML_E^{(\n)}$ whenever $\ll\in\Nl_{(\n)}^r.$
 
On the other hand, the tripotent $c=e_{[j]}=e_1+\;+e_j\in\c E_j$ satisfies
$$\x{ord}_c(N_m)=\x{ord}_0(N_m(c+\.))=\begin{cases}m-j&m>j\\0&m\le j\end{cases}.$$
It follows that
\be{13}\x{ord}_c(N^\ll)=\S_{m=1}^r(\ll_m-\ll_{m+1})\ \x{ord}_c(N_m)=\S_{m=j+1}^r(\ll_m-\ll_{m+1})\ \x{ord}_c(N_m)$$
$$=\S_{m=j+1}^r(\ll_m-\ll_{m+1})(m-j)=\S_{m=j+1}^r\ll_m.\ee
Thus if $\ll$ is a partition such that $J^\ll\ic\ML_E^{(\n)}$ then $N^\ll\in\ML_E^{(\n)}$ and \er{13} implies
$$\S_{m=j+1}^r\ll_m=\x{ord}_c(N^\ll)\ge n_{j+1}$$
for all $0\le j\le r-1.$ Therefore $\ll\in\Nl_{(\n)}^r.$
\end{proof}

The set $\lL$ determined in Theorem \ref{e} is not minimal. For example $\n=(10,5,1)$ has the minimal partitions $(5,4,1)$ and 
$(5,3,2),$ whereas $\n=(15,5,1)$ has the minimal partitions $(10,4,1),(9,5,1),(8,6,1),(10,3,2),(9,4,2),(8,5,2),(9,3,3)$ and 
$(8,4,3).$

We next study the normal projection map \er{5} for determinantal ideals. Given $\m=(m_{\l+1},\;,m_r)\in\Nl_+^{n-\l}$ we put
$$\Nl^{r-\l}_{(\m)}:=\{\la\in\Nl_+^{r-\l}:\ \la_j+\;+\la_r\ge m_j\ \forall\ \l<j\le r\}.$$
Since $\x{rank}_E(c+w)=\l+\x{rank}_W(w)$ it follows that
$$\h E_{j-1}\ui(c+W)=\begin{cases}\OO&j\le\l\\c+\h W_{j-\l-1}&j>\l\end{cases}.$$

\begin{theorem} Let $c$ be a tripotent of rank $\l.$ For $\n\in\Nl_+^r$ put
$$\n^*:=(n_{\l+1},\;,n_r)\in\Nl_+^{r-\l}.$$
Then the normal projection map $\lp_c$ satisfies
$$\ML_E^{(\n)}\xrightarrow{\lp_c}\ML_W^{(\n^*)}.$$
In particular, for $j>\l$
\be{16}\ML_{\h E_{j-1}}^{(n)}\xrightarrow{\lp_c}\ML_{\h W_{j-\l-1}}^{(n)}.\ee
\end{theorem}
\begin{proof} By Theorem \ref{e} we may assume that $f\in\ML_E^{(\n)}$ belongs to $J^\ll$ for a partition 
$\ll=(\ll_1,\;,\ll_r)$ satisfying 
\be{15}\ll_j+\;+\ll_r\ge n_j\ee 
for all $1\le j\le r.$ Now fix $\l$ and consider the partition
$\ll^*:=(\ll_{\l+1},\;,\ll_r)$ of length $r-\l.$ By \er{14} we have
$$\lp_c f\in\S_{\la\ge\ll^*}J_W^\la$$
where $\la\ge\ll^*$ is the (partial) containment order, i.e. $\la_i\ge\ll_i$ for all $\l<i\le r.$ For $\l<j\le r$ we have
$$\la_j+\;+\la_r\ge\ll_j+\;+\ll_r\ge n_j.$$
Therefore $\la$ satisfies the analogue of \er{15} relative to $W.$ Applying Theorem \ref{e} to $W$ and the sequence 
$n_{\l+1}\ge\;\ge n_r$ of length $r-\l=\x{rank}(W),$ it follows that $\lp_c J^\ll\ic\ML_W^{(\n^*)}.$ 
The special case \er{16} corresponds to $n_1=\;=n_j=n,\ n_{j+1}=\;=n_r=0.$
\end{proof}

We now consider the special case of "step 1" partitions. Let $n\in\Nl$ and consider the ideal
\be{17}\ML_{\h E_\l}^{(n)}=\{p\in\PL_E:\ \x{ord}_\lz(p)\ge n\ \forall\ \lz\in\h E_\l\},\ee
where $0\le\l\le r$ is fixed. This corresponds to $\n=(n^{(\l+1)},0^{(r-\l-1)}).$
In this case Theorem \ref{e} yields
$$\ML_{\h E_\l}^{(n)}=\S_{\ll_{\l+1}+\;+\ll_r\ge n}J^\ll.$$
For example 
$$\ML_{\h E_0}^{(n)}=\ML_0^n=\S_{|\ll|=\ll_1+\;+\ll_r\ge n}J^\ll$$
and
$$\ML_{\h E_{r-1}}^{(n)}=\S_{\ll_r\ge n}J^\ll.$$

For the $K$-invariant ideals \er{17} it is easy to find a minimal decomposition:

\begin{theorem} For $0\le\l<r$ the ideal $\ML_{\h E_\l}^{(n)}$ is the minimal and finite sum
$$\ML_{\h E_\l}^{(n)}=\S_\la J^{\h\la}$$
taken over the (finitely many) integer tuples $\la_{\l+1}\ge\;\ge\la_r\ge 0$ satisfying 
$$|\la|:=\la_{\l+1}+\;+\la_r=n.$$ 
Here we put 
$$\h\la:=(\la_{\l+1}^{(\l)},\la_{\l+1},\;,\la_r)\in\Nl_+^r.$$
\end{theorem}
\begin{proof} We first show that the sum is minimal. If $\la,\lb\in\Nl_+^{r-\l}$ satisfy $|\la|=n=|\lb|$ and 
$\h\la\ge\h\lb,$ then $\la\ge\lb$ and hence $\la=\lb,$ showing that $\h\la=\h\lb.$ Of course there are only finitely many partitions $\la=(\la_{\l+1},\;,\la_r)$ satisfying $|\la|=n.$ By Theorem \ref{e} $\ML_{\h E_\l}^{(n)}$ corresponds to partitions 
$\ll\in\Nl_+^r$ such that the single inequality
\be{18}\ll_{\l+1}+\;+\ll_r\ge n\ee
holds. Clearly, $\ll=\h\la$ satisfies \er{18}, since $\ll_i\ge\la_i$ for $\l<i\le r.$ Thus it remains to show that \er{18} implies
$\ll\ge\h\la$ for some $\la\in\Nl_+^{r-\l}$ with $|\la|=n.$ If $\ll_{\l+1}+\;+\ll_r\ge n$ then there exists 
$\la_{\l+1}\ge\;\ge\la_r$ such that $|\la|=n$ and $\ll_i\ge\la_i$ for $\l<i\le r.$ To see this, we may assume (by induction) that 
$\ll_{\l+1}+\;+\ll_r=n+1.$ Let 
$$\ll_{\l+1}\ge\;\ge\ll_m>0=\ll_{m+1}=\;=\ll_r,$$
where $\l<m\le r.$ Then $\la:=(\ll_{\l+1},\;,\ll_{m-1},\ll_m-1,0,\;,0)$ is decreasing, satisfies $|\la|=n$ and 
$\ll\ge\h\la$ since for $1\le i\le\l$ we have 
$\h\la_i=\begin{cases}\ll_{\l+1}&m>\l+1\\\ll_{\l+1}-1&m=\l+1\end{cases}$ and therefore $\h\la_i\le\ll_{\l+1}\le\ll_i.$
\end{proof}

In the simplest case $n=1$ there is only a single choice $\la_{\l+1}=1,\ \la_{\l+2}=\;=\la_r=0$ yielding the 
"fundamental partition" $\h\la=(1^{(\l+1)},0^{(r-\l-1)})\equiv 1^{(\l+1)}.$ Thus the prime ideal 
$\ML_{\h E_\l}^{(1)}=\ML_{\h E_\l}$ has the form
$$\ML_{\h E_\l}^{(1)}=J^{\h\la}=J^{1^{(\l+1)}}.$$
For $n>1$ we cannot represent $\ML_{\h E_\l}^{(n)}$ by a single partition.

\section{Reproducing Kernels and Hermitian Structure}
The main result of \cite{U3}, formulated as the isomorphism \er{19}, determines the localization bundle $J_E^\ll$ as an abstract 
(singular) holomorphic vector bundle, without reference to a hermitian metric. If $H=\o I$ is a Hilbert module completion of a 
$K$-invariant ideal $I,$ with reproducing kernel function $\KL(z,\lz),$ the corresponding eigenbundle $\y H\ap\y I|_D$ carries the all-important {\bf hermitian structure} as a subbundle of $D\xx H.$ To make the connection one needs an explicit embedding 
$$\y I|_D\xrightarrow{}\y H\ic D\xx H.$$ 
As suggested by the "jet construction" developed in \cite{DMV} for smooth submanifolds, such a map should involve certain derivatives of $\KL(z,\lz)$ in the normal direction. In this section we carry out this program in the more complicated geometric situation related to an arbitrary partition ideal $I=J^\ll.$

Any polynomial $p$ induces a constant coefficient differential operator $\o p(\dl_z)$ on $E,$ depending in a conjugate-linear way on $p.$ Let $(p|q)$ denote the Fischer-Fock inner product of polynomials $p,q\in\PL_E$ (anti-linear in $p$) and let
$\EL^\ll(z,\lz)=\EL_\lz^\ll(z)$ denote the reproducing kernel of $\PL_E^\ll.$

Every irreducible $J^*$-triple $E$ has two "characteristic multiplicities" $a,b$ \cite{Ar,U1,L1,L2} such that
$$\f dr=1+\f a2(r-1)+b$$

\begin{lemma} For a Jordan triple $E,$ the determinant function $N_e$ at a maximal tripotent $e\in S_r$ satisfies
\be{20}\EL^\ll(z,\lz)=C_r^n(\ll)\ N_e^n(z)\o{N_e(\lz)}^n\ \EL^{\ll-n^{(r)}}(z,\lz)\ee
for $\ll\ge n^{(r)}$ and $\lz\in E_e,$ where
$$C_r^n(\ll)=\P_{j=1}^r\f1{(\ll_j-n+1+\tfrac a2(r-j))_n}.$$
Moreover,
\be{21}\o N_e^n(\dl_z)\EL_\lz^\ll=\o{N_e(\lz)}^n\ \EL_\lz^{\ll-n^{(r)}}.\ee
\end{lemma}
\begin{proof} We use the Jordan theoretic Pochhammer symbols $(s)_\ll$ and the "Faraut-Kor\'anyi formula" \cite{FK1,FK2}. Suppose first that $E=E_e$ is unital, with unit element $e$ and determinant $\lD_e.$ The parameter $s=d/r$ in the continuous Wallach set corresponds to the {\bf Hardy space} $H^2(S)$ over the Shilov boundary $S$ of $D.$ Since $|\lD_e|=1$ on $S$ we obtain for $p,q\in\PL_E^\ll$
$$(\o\lD_e^n(\dl_z)(\lD_e^np)|q)=(\lD_e^np|\lD_e^nq)=(d/r)_{\ll+n^{(r)}}\ (\lD_e^np|\lD_e^nq)_S$$
$$=(d/r)_{\ll+n^{(r)}}\ (p|q)_S=\f{(d/r)_{\ll+n^{(r)}}}{(d/r)_\ll}\ (p|q).$$
Since $q$ is arbitrary, it follows that
$$\o\lD_e^n(\dl_z)(\lD_e^n p)=\f{(d/r)_{\ll+n^{(r)}}}{(d/r)_\ll}\ p$$
for all partitions $\ll$ and $p\in\PL_E^\ll.$ An application of Schur orthogonality \cite[Proposition XI.4.1]{FK2} shows that
$$\EL^\ll(e,e)=\f{d_\ll}{(d/r)_\ll},$$
where $d_\ll:=\dim\PL_E^\ll.$ If $\ll\ge n^{(r)}$ we can write 
$$\EL^\ll(z,\lz)=a_\ll\ \lD_e^n(z)\o{\lD_e(\lz)}^n\ \EL^{\ll-n^{(r)}}(z,\lz)$$ 
for some coefficient $a_\ll.$ It follows that
$$\f{d_\ll}{(d/r)_\ll}=\EL^\ll(e,e)=a_\ll\ \lD_e^n(e)\o{\lD_e(e)}^n\ \EL^{\ll-n^{(r)}}(e,e)$$
$$=a_\ll\ \EL^{\ll-n^{(r)}}(e,e)=a_\ll\ \f{d_{\ll-n^{(r)}}}{(d/r)_{\ll-n^{(r)}}}.$$
Since $d_\ll=d_{\ll-n^{(r)}}$ in the unital case it follows that
$$a_\ll=\f{d_\ll}{d_{\ll-n^{(r)}}}\ \f{(d/r)_{\ll-n^{(r)}}}{(d/r)_\ll}=\f{(d/r)_{\ll-n^{(r)}}}{(d/r)_\ll}=C_r^n(\ll).$$
This proves \er{20}. Moreover,
$$\f1{C_r^n(\ll)}\o\lD_e^n(\dl_z)\EL_\lz^\ll=\o\lD_e^n(\dl_z)\(\lD_e^n\ \o{\lD_e(\lz)}^n\ \EL_\lz^{\ll-n^{(r)}}\)$$
$$=\o{\lD_e(\lz)}^n\ \o\lD_e^n(\dl_z)(\lD_e^n\ \EL_\lz^{\ll-n^{(r)}})
=\o{\lD_e(\lz)}^n\ \f1{C_r^n(\ll)}\ \EL_\lz^{\ll-n^{(r)}}=\f1{C_r^n(\ll)}\o{\lD_e(\lz)}^n\ \EL_\lz^{\ll-n^{(r)}}.$$

In the non-unital case, let $P_e$ be the Peirce 2-projection onto $E_e.$ Since $N_e(z)=\lD_e(P_ez)$ by definition, applying the unital case to $E_e$ and using $P_e\lz=\lz$ we obtain
$$\EL^\ll(z,\lz)=\EL^\ll(P_ez,\lz)=C_r^n(\ll)\ \lD_e^n(P_ez)\o{\lD_e(\lz)}^n\ \EL^{\ll-n^{(r)}}(P_ez,\lz)$$
$$=C_r^n(\ll)\ N_e^n(z)\o{N_e(\lz)}^n\ \EL^{\ll-n^{(r)}}(z,\lz).$$
The second assertion follows with $\o\lD_e^n(\dl_z)f=\o N_e^n(\dl_z)(f\oc P_e).$
\end{proof}

The identity \er{21}, written as
$$\o N_e^n(\dl_z)\EL^\ll(z,\lz)=\o{N_e(\lz)}^n\ \EL^{\ll-n^{(r)}}(z,\lz),$$
implies
$$N_e^n(\o\dl_\lz)\EL^\ll(z,\lz)=N_e^n(z)\ \EL^{\ll-n^{(r)}}(z,\lz),$$
since $\o{\EL^\ll(z,\lz)}=\EL^\ll(\lz,z).$ Thus for $\lz\in E_e$ we have
\be{22}N_e^n(\o\dl_\lz)\EL_\lz^\ll=N_e^n\ \EL_\lz^{\ll-n^{(r)}}\ee
as holomorphic polynomials in $z.$

\begin{lemma}\label{f} Let $\ll\in\Nl_+^r$ be a partition and $\lz\in E$ have rank $\l.$ Then $\EL_\lz^\ll\ne0$ if and only if 
$\ll$ has length $\le\l.$
\end{lemma}
\begin{proof} If $\ll$ has length $>\l$ then all $p\in\PL_E^\ll$ vanish on $\h E_\l.$ Therefore 
$$(\EL_\lz^\ll|p)=p(\lz)=0$$
and hence $\EL_\lz^\ll=0$ since $p$ is arbitrary. Now suppose $\ll$ has length $\le\l.$ Consider the spectral decomposition
$$\lz=\S_{i=1}^\l\lz_i e_i$$ 
for a frame $(e_i).$ The conical polynomial $N^\ll$ defined as in \er{23} satisfies
$$N^\ll(\lz)=\lz_1^{\ll_1}\:\lz_\l^{\ll_\l}\ne 0.$$
Since $N^\ll(\lz)=(\EL_\lz^\ll|N^\ll)$ it follows that $\EL_\lz^\ll\ne 0.$ 
\end{proof}

For $n\in\Nl,$ define $n^{(m)}=(n,\;,n,0,\;,0),$ with $n$ repeated $m$ times. Thus the Young diagram $[n^{(m)}]=[1,m]\xx[1,n].$ Any partition $\ll$ can be written as
\be{24}\ll=(n_1^{(\l_1)},n_2^{(\l_2-\l_1)},\;,n_t^{(\l_t-\l_{t-1})},0^{(r-\l_t)}),\ee
where $1\le \l_1<\;<\l_t\le r$ and $n_1>n_2>\;>n_t>0.$ Thus $t$ is the number of "steps" in the partition. In other words,
$$\ll_1=\;=\ll_{\l_1}=n_1>\ll_{1+\l_1}=\;=\ll_{\l_2}=n_2>\ll_{1+\l_2}\;$$
Define for $1\le s\le t$ 
$$\lO_s:=\{\lz\in D:\ \l_{s-1}\le\x{rank}(\lz)<\l_s\}=D\ui\U_{i=\l_{s-1}}^{\l_s-1}\c E_\l.$$
We also consider the "regular" points
$$\c D:=\{\lz\in D:\ \x{rank}(\lz)\ge\l_t\}.$$
Since $\ll$ has length $\l_t,$ Lemma \ref{f} implies $\EL_\lz^\ll\ne 0$ if $\x{rank}(\lz)\ge\l_t.$ Hence $\KL_\lz$ does not vanish at $\lz\in\c D.$ For the "singular" points $\lz\in D$ there exists a unique $s\le t$ such that $\lz\in\lO_s.$ For 
$1\le h\le k\le t$ define
$$\ll_h^k:=(n_h^{(\l_h-\l_{h-1})},n_{h+1}^{(\l_{h+1}-\l_h)},\;,n_k^{(\l_k-\l_{k-1})})$$
as a partition of length $\l_k-\l_{h-1}.$ Let $\lm\in\Nl_+^{\l_s}$ satisfy $\lm\ge\ll_1^s.$ For each $s\le k\le t$ consider the partition
$$(\lm,\ll_{s+1}^k)=(\lm,n_{s+1}^{(\l_{s+1}-\l_s)},n_{s+2}^{(\l_{s+2}-\l_{s+1})},\;,n_k^{(\l_k-\l_{k-1})})$$
of length $\l_k.$ Here $\ll_{s+1}^k$ is empty for $s=k.$ Starting with the Peter-Weyl expansion
$$\KL(z,\lz)=\S_{\lm\ge\ll}a_\lm\ \EL^\lm(z,\lz)$$
of the reproducing kernel $\KL(z,\lz)$ of $H=\o J^\ll,$ with coefficients $a_\lm>0,$ we define $K$-invariant sesqui-holomorphic functions 
$\KL^s(z,\lz)=\KL_\lz^s(z)$ on $D\xx D$ by the Peter-Weyl expansion
\be{25}\KL_\lz^s=\S_{\Nl_+^{\l_s}\ni\lm\ge\ll_1^s}a_{(\lm,\ll_{s+1}^t)}\ \EL_\lz^{\lm-n_s^{(\l_s)}}
\P_{k=s}^t C_{\l_k}^{n_k-n_{k+1}}((\lm,\ll_{s+1}^k)-n_{k+1}^{(\l_k)}).\ee
More explicitly, the constant is given by
$$C_{\l_t}^{n_t}(\lm,\ll_{s+1}^t)\.C_{\l_{t-1}}^{n_{t-1}-n_t}((\lm,\ll_{s+1}^{t-1})-n_t^{(\l_{t-1})})
\.C_{\l_{t-2}}^{n_{t-2}-n_{t-1}}((\lm,\ll_{s+1}^{t-2})-n_{t-1}^{(\l_{t-2})})$$
$$\:C_{\l_{s+1}}^{n_{s+1}-n_{s+2}}((\lm,n_{s+1}^{(\l_{s+1}-\l_s)})-n_{s+2}^{(\l_{s+1})})
\.C_{\l_s}^{n_s-n_{s+1}}(\lm-n_{s+1}^{(\l_s)}).$$

\begin{lemma} The kernel $\KL_\lz^s$ does not vanish if $\x{rank}(\lz)\ge\l_{s-1}$ and vanishes if $\x{rank}(\lz)<\l_{s-1}.$ Thus $\KL^s$ vanishes precisely on $\o\lO_{s-1}=\U_{0\le k<s}\lO_k.$
\end{lemma}
\begin{proof} Since the partition $\ll_1^s-n_s^{(\l_s)}=\ll_1^{s-1}$ has length $\l_{s-1}$ it follows that 
$\EL_\lz^{\ll_1^s-n_s^{(\l_s)}}\ne 0$ if $\x{rank}(\lz)\ge\l_{s-1}.$ This implies $\KL_\lz^s\ne 0$ since 
$\EL^{\ll_1^s-n_s^{(\l_s)}}$ occurs as the lowest term in $\KL^s.$ On the other hand, if $\x{rank}(\lz)<\l_{s-1}$ then 
$\EL_\lz^{\lm-n_s^{(\l_s)}}=0$ for all $\lm\ge\ll_1^s$ since $\lm-n_s^{(\l_s)}\ge\ll_1^s-n_s^{(\l_s)}=\ll_1^{s-1}$ has length 
$\ge\l_{s-1}.$
\end{proof}

\begin{proposition}\label{g} Let $\x{rank}(\lz)\le\l_t$ and write $\lz=\lim_{\Le\to 0}\lz_\Le,$ where $\lz_\Le\in\c E_{c_t}$ for some tripotent $c_t\in S_{\l_t}.$ Then $N_{c_t}(\lz_\Le)\ne 0$ and
$$\lim_{\Le\to0}\f{\KL_{\lz_\Le}}{\o{N_{c_t}(\lz_\Le)}^{n_t}}=N_{c_t}^{n_t}\ \KL_\lz^t.$$
Similarly, for $1\le s<t$ let $\x{rank}(\lz)\le\l_s$ and write $\lz=\lim_{\Le\to 0}\lz_\Le,$ where $\lz_\Le\in\c E_{c_s}$ for some tripotent $c_s\in S_{\l_s}.$ Then $N_{c_s}(\lz_\Le)\ne 0$ and
$$\lim_{\Le\to0}\f{\KL_{\lz_\Le}^{s+1}}{\o{N_{c_s}(\lz_\Le)}^{n_s-n_{s+1}}}=N_{c_s}^{n_s-n_{s+1}}\ \KL_\lz^s.$$
\end{proposition}
\begin{proof} Since $\lz_\Le$ has rank $\l_t,$ it follows that $\EL_{\lz_\Le}^\ll=0$ unless 
$\ll=(\lm,0^{(r-\l_t)}),$ where $\lm\in\Nl_+^{\l_t}$ satisfies $\lm\ge\ll_1^t.$ Therefore
$$\KL_{\lz_\Le}=\S_{\Nl_+^{\l_t}\ni\lm\ge\ll_1^t}a_{(\lm,0^{(r-\l_t)})}\ \EL_{\lz_\Le}^\lm.$$
Since $\lm\ge n_t^{(\l_t)},$ we may apply \er{20} to $E_{c_t}$ and $n=n_t.$ This yields
$$\f{\KL_{\lz_\Le}}{\o{N_{c_t}(\lz_\Le)}^{n_t}}
=\S_{\Nl_+^{\l_t}\ni\lm\ge\ll_1^t}a_{(\lm,0^{(r-\l_t)})}\ \f{\EL_{\lz_\Le}^\lm}{\o{N_{c_t}(\lz_\Le)}^{n_t}}
=N_{c_t}^{n_t}\S_{\Nl_+^{\l_t}\ni\lm\ge\ll_1^t}a_{(\lm,0^{(r-\l_t)})}\ C_{\l_t}^{n_t}(\lm)\ \EL_{\lz_\Le}^{\lm-n_t^{(\l_t)}}.$$
It follows that
$$\lim_{\Le\to0}\f{\KL_{\lz_\Le}}{\o{N_{c_t}(\lz_\Le)}^{n_t}}
=N_{c_t}^{n_t}\S_{\Nl_+^{\l_t}\ni\lm\ge\ll_1^t}a_{(\lm,0^{(r-\l_t)})}\ C_{\l_t}^{n_t}(\lm)\ \EL_\lz^{\lm-n_t^{(\l_t)}}
=N_{c_t}^{n_t}\ \KL_\lz^t.$$

Now let $s<t.$ Let $\lm\in\Nl_+^{\l_{s+1}}$ satisfy $\lm\ge\ll_1^{s+1}.$ Since $\x{rank}(\lz_\Le)=\l_s$ we have 
$\EL_{\lz_\Le}^{\lm-n_{s+1}^{(\l_{s+1})}}=0$ unless $\lm=(\ln,n_{s+1}^{(\l_{s+1}-\l_s)}),$ where $\ln\in\Nl_+^{\l_s}$ satisfies 
$\ln\ge\ll_1^s.$ Applying definition \er{25} to $\KL^{s+1}$ yields
$$\KL_{\lz_\Le}^{s+1}=\S_{\ll_1^{s+1}\le\lm\in\Nl_+^{\l_{s+1}}}a_{(\lm,\ll_{s+2}^t)}\ \EL_\lz^{\lm-n_{s+1}^{(\l_{s+1})}}
\P_{k=s+1}^t C_{\l_k}^{n_k-n_{k+1}}((\lm,\ll_{s+2}^k)-n_{k+1}^{(\l_k)})$$
$$=\S_{\ll_1^s\le\ln\in\Nl_+^{\l_s}}a_{(\ln,\ll_{s+1}^t)}\ \EL_\lz^{\ln-n_{s+1}^{(\l_s)}}
\P_{k=s+1}^t C_{\l_k}^{n_k-n_{k+1}}((\ln,\ll_{s+1}^k)-n_{k+1}^{(\l_k)}),$$
using $(\ln,n_{s+1}^{(\l_{s+1}-\l_s)},\ll_{s+2}^t)=(\ln,\ll_{s+1}^t)$ and 
$(\ln,n_{s+1}^{(\l_{s+1}-\l_s)})-n_{s+1}^{(\l_{s+1})}=\ln-n_{s+1}^{(\l_s)}.$ Since $\ln\ge(n_s-n_{s+1})^{(\l_s)},$ we may apply \er{20} to $E_{c_s}$ and $n=n_s-n_{s+1}.$ Since $\ln-n_{s+1}^{(\l_s)}-(n_s-n_{s+1})^{(\l_s)}=\ln-n_s^{(\l_s)}$ this yields
$$\f{\KL_{\lz_\Le}^{s+1}}{\o{N_{c_s}(\lz_\Le)}^{n_s-n_{s+1}}}=\S_{\ll_1^s\le\ln\in\Nl_+^{\l_s}}a_{(\ln,\ll_{s+1}^t)}
\ \f{\EL_{\lz_\Le}^{\ln-n_{s+1}^{(\l_s)}}}{\o{N_{c_s}(\lz_\Le)}^{n_s-n_{s+1}}}
\P_{k=s+1}^t C_{\l_k}^{n_k-n_{k+1}}((\ln,\ll_{s+1}^k)-n_{k+1}^{(\l_k)})$$
$$=N_{c_s}^{n_s-n_{s+1}}\S_{\ll_1^s\le\ln\in\Nl_+^{\l_s}}a_{(\ln,\ll_{s+1}^t)}
\ \EL_{\lz_\Le}^{\ln-n_s^{(\l_s)}}\ C_s^{n_s-n_{s+1}}(\ln-n_{s+1}^{(\l_s)})
\P_{k=s+1}^t C_{\l_k}^{n_k-n_{k+1}}((\ln,\ll_{s+1}^k)-n_{k+1}^{(\l_k)})$$
$$=N_{c_s}^{n_s-n_{s+1}}\S_{\ll_1^s\le\ln\in\Nl_+^{\l_s}}a_{(\ln,\ll_{s+1}^t)}\ \EL_{\lz_\Le}^{\ln-n_s^{(\l_s)}}
\P_{k=s}^t C_{\l_k}^{n_k-n_{k+1}}((\ln,\ll_{s+1}^k)-n_{k+1}^{(\l_k)}).$$
It follows that
$$\lim_{\Le\to0}\f{\KL_{\lz_\Le}^{s+1}}{\o{N_{c_s}(\lz_\Le)}^{n_s-n_{s+1}}}
=N_{c_s}^{n_s-n_{s+1}}\S_{\ll_1^s\le\ln\in\Nl_+^{\l_s}}a_{(\ln,\ll_{s+1}^t)}\ \EL_\lz^{\ln-n_s^{(\l_s)}}
\P_{k=s}^t C_{\l_k}^{n_k-n_{k+1}}((\ln,\ll_{s+1}^k)-n_{k+1}^{(\l_k)})$$
$$=N_{c_s}^{n_s-n_{s+1}}\ \KL_\lz^s.$$
\end{proof}

Let $1\le s\le t.$ For any tripotent chain 
$$c_s<c_{s+1}<\;<c_t$$
of rank $\l_s<\l_{s+1}<\;<\l_t$ consider the polynomial
$$N_{c_s,\;,c_t}^{n_s,\;,n_t}:=\P_{k=s}^t N_{c_k}^{n_k-n_{k+1}}=N_{c_s}^{n_s-n_{s+1}}N_{c_{s+1}}^{n_{s+1}-n_{s+2}}\:N_{c_{t-1}}^{n_{t-1}-n_t}N_{c_t}^{n_t}.$$

\begin{theorem}\label{h} Let $1\le s\le t.$ Then for any $\lz\in\lO_s$ we have 
$$H_\lz=\<N_{c_s,\;,c_t}^{n_s,\;,n_t}\.\KL_\lz^s:\ c_s<c_{s+1}<\;<c_t,\ \lz\in E_{c_s}\>$$
where the linear span is taken over all tripotent chains of rank $\l_s<\l_{s+1}<\;<\l_t$ such that $\lz\in E_{c_s}.$ In view of \er{21} we can write $N_{c_s,\;,c_t}^{n_s,\;,n_t}\.\KL_\lz^s$ also in the differential form
$N_{c_s,\;,c_t}^{n_s,\;,n_t}(\o\dl_\lz)\t\KL_\lz^s$ for some modified kernel functions $\t\KL_\lz^s.$ 
\end{theorem}
\begin{proof} We first show that $N_{c_s,\;,c_t}^{n_s,\;,n_t}\.\KL_\lz^s\in H_\lz$ for all $\lz\in E_{c_s}.$ In case $s=t$ let $c_t\in S_{\l_t}$ and $\lz\in E_{c_t}.$ Writing $\lz=\lim_{\Le\to0}\lz_\Le$ for some curve 
$\lz_\Le\in\c E_{c_t}$ we have
$$\f{\KL_{\lz_\Le}}{\o{N_{c_t}(\lz_\Le)}^{n_t}}\xrightarrow[\Le\to0]{}N_{c_t}^{n_t}\ \KL_\lz^t$$
by Proposition \ref{g}. Since $\KL_{\lz_\Le}\in H_{\lz_\Le}$ a continuity argument yields $N_{c_t}^{n_t}\ \KL_\lz^t\in H_\lz.$
Now let $s<t$ and $\lz\in E_{c_s}.$ Writing $\lz=\lim_{\Le\to0}\lz_\Le$ for some curve $\lz_\Le\in\c E_{c_s}$ we have
$$\f{\KL_{\lz_\Le}^{s+1}}{\o{N_{c_s}(\lz_\Le)}^{n_s-n_{s+1}}}\xrightarrow[\Le\to0]{}N_{c_s}^{n_s-n_{s+1}}\ \KL_\lz^s$$
by Proposition \ref{g}. Hence
$$\f{N_{c_{s+1},\;,c_t}^{n_{s+1},\;,n_t}\KL_{\lz_\Le}^{s+1}}{\o{N_{c_s}(\lz_\Le)}^{n_s-n_{s+1}}}
\xrightarrow[\Le\to0]{}N_{c_{s+1},\;,c_t}^{n_{s+1},\;,n_t}N_{c_s}^{n_s-n_{s+1}}\ \KL_\lz^s
=N_{c_s,\;,c_t}^{n_s,\;,n_t}\ \KL_\lz^s.$$
By (downward) induction we may assume $N_{c_{s+1},\;,c_t}^{n_{s+1},\;,n_t}\KL_{\lz_\Le}^{s+1}\in H_{\lz_\Le}$ since 
$\lz_\Le\in E_{c_s}\ic E_{c_{s+1}}.$ A continuity argument shows that $N_{c_s,\;,c_t}^{n_s,\;,n_t}\ \KL_\lz^s\in H_\lz,$ concluding the induction step.

Now let $\lz\in\lO_s$ so that $\l_{s-1}\le\l=\x{rank}(\lz)<\l_s.$ Consider all tripotent chains $c_s<c_{s+1}<\;<c_t$ of rank 
$\l_s<\l_{s+1}<\;<\l_t$ such that $\lz\in E_{c_s}$ and denote by $\lS\ic\PL_E$ the linear span of the associated polynomials $N_{c_s,\;,c_t}^{n_s,\;,n_t}.$ Since the kernel $\KL_\lz^s$ does not vanish by Lemma \ref{g} multiplication by $\KL_\lz^s$ is an injective map $\lS\to H_\lz.$ To show surjectivity, assume first that $\lz=c$ is a tripotent such that $c<c_s.$ Then
$$\lp_c N_{c_s,\;,c_t}^{n_s,\;,n_t}=\t N_{c_s-c,\;,c_t-c}^{n_s,\;,n_t}$$
where $c_s-c<\;<c_t-c$ is a tripotent chain of rank $\l_s-\l<\;<\l_t-\l$ in $W=E^c$ and $\t N$ denotes the conical polynomials relative to $W.$ Since the polynomials $\t N_{c_s-c,\;,c_t-c}^{n_s,\;,n_t}$ span $\PL_W^{\ll^*}$ it follows that
$$\dim\lS\ge\dim\PL_W^{\ll^*}=\dim J_c^\ll=\dim H_c.$$
Here we use \er{19} for the second equation. Thus it follows that $H_c=\lS\.\KL_c^s.$ In the general case let $\lz\in E_{c_s}$ have rank $\l<\l_s.$ Then there exists $h\in\h K$ such that $\lz=h^*c,$ where $c\in E_{c_s}$ is a tripotent of rank $\l$ with 
$c<c_s.$ Moreover, we may assume that $h E_{c_j}=E_{c_j}$ for all $s\le j\le t.$ By semi-invariance we have 
$N_{c_j}\oc h^{-1}=a_j\ N_{c_j}$ for some constant $a_j.$ This implies
$$N_{c_s,\;,c_t}^{n_s,\;,n_t}\oc h^{-1}=\P_{j=s}^t a_j^{n_j-n_{j+1}}\ N_{c_s,\;,c_t}^{n_s,\;,n_t}.$$
On the other hand, the $K$-invariance of $\KL^s$ implies $\KL_\lz^s=\KL_{h^*c}^s=\KL_c^s\oc h.$ It follows that 
$$N_{c_s,\;,c_t}^{n_s,\;,n_t}\ \KL_\lz^s=\((N_{c_s,\;,c_t}^{n_s,\;,n_t}\oc h^{-1})\ \KL_c^s\)\oc h
=\P_{j=s}^t a_j^{n_j-n_{j+1}}(N_{c_s,\;,c_t}^{n_s,\;,n_t}\ \KL_c^s)\oc h.$$
Using $H_\lz=H_{h^*c}=H_c\oc h$ the assertion follows in the general case.
\end{proof}

If $\l_{s-1}\le\l<\l_s$ and $\ll=(\ll_1,\;,\ll_r)=(n_1^{(\l_1-\l_2)},\;,n_t^{(\l_t)},0^{(r-\l_t)}),$ we put 
$$\ll^*=(\ll_{\l+1},\;,\ll_r)=(n_s^{(\l_s-\l)},n_{s+1}^{(\l_{s+1}-\l_s)}\;,n_t^{(\l_t)},0^{(r-\l_t)})$$
and $\h\ll^*=(\ll_{\l+1}^{(\l)},\ll^*).$ For any tripotent $c$ of rank $\l,$ with Peirce 0-space $W=E^c,$ there exists a cross-section \cite{U3} $\Lt_{\h\ll^*}:\PL_W^{\ll^*}\to\PL_E^{\h\ll^*}$ satisfying
$$\Lt_{\h\ll^*}\t N_{c_s-c,\;,c_t-c}^{n_s,\;,n_t}=N_{c_s,\;,c_t}^{n_s,\;,n_t}$$
for all tripotent chains $c_s<\;<c_t$ of rank $\l_s<\;<\l_t$ such that $c<c_s.$ Thus Theorem \ref{h} and its proof yield the 
"embedding theorem"

\begin{theorem} For each tripotent $c\in\lO_s$ there is an isomorphism 
$$\PL_W^{\ll^*}\xrightarrow[\ap]{\lT_c}H_c\ic H$$ 
mapping $\lf\in\PL_W^{\ll^*}$ to $(\Lt_{\h\ll^*}\lf)\.\KL_c^s\in H_c.$
\end{theorem}

Beyond the tripotents, for any $\lz\in\lO_s$ we define the embedding $J_\lz^\ll\xrightarrow[\ap]{\lT_\lz}H_\lz$ via the commuting diagram
$$\xymatrix{J_c^\ll\ar[r]^{\lT_c}\ar[d]_{\oc h^{-1}}&H_c\ar[d]^{\oc h^*}\\
J_\lz^\ll\ar[r]_{\lT_\lz}&H_\lz}$$
where $h\in\h K$ satisfies $\lz=hc.$

\section{Geometric Realization} 
The famous results of Kodaira \cite{MK} show that a compact K\"ahler manifold $M,$ whose K\"ahler form $\lo$ satisfies an integrality condition (a so-called Hodge manifold) can be embedded as an algebraic subvariety of a projective space $\Pl(V).$ The underlying finite-dimensional vector space $V$ (more precisely its linear dual space) has a geometric realization in terms of holomorphic sections of a suitable line bundle $\LL$ over $M.$ For homogeneous compact K\"ahler manifolds, the  Borel-Weil-Bott theorem \cite{Ac} gives a similar construction for the irreducible representations of compact semisimple Lie groups. In this section we show that such a "geometric" description also holds for the fibres of the localization bundle $J_E^\ll,$ for any 
"partition" ideal $J^\ll.$ While this realization, based on the complex geometry of Peirce 2-spaces, still needs Jordan theory, the underlying idea, using flags of certain subspaces ("slices") of $E,$ may be susceptible to generalization.

The collection $M_\l$ of all Peirce 2-spaces $U=E_c,$ for all tripotents $c\in S_\l$ of fixed rank $\l,$ is a compact K\"ahler manifold called the $\l$-th {\bf Peirce manifold}. In the matrix case $E=\Cl^{r\xx s}$ and $0\le\l\le r$ we may identify 
$M_\l\ap\x{Grass}_\l(\Cl^r)\xx\x{Grass}_\l(\Cl^s)$ as a product of Grassmann manifolds.

The structure group $\h K$ acts transitively on each $M_\l$ by holomorphic transformations $(h,U)\qi hU,$ and the restricted action of $K$ is already transitive. It follows that
$$M_\l=\h K/\h K^{E_c}=K/K^{E_c}$$
where $\h K^{E_c}:=\{\lg\in\h K:\ \lg E_c=E_c\}$ is the closed complex subgroup fixing a given Peirce 2-space $E_c.$

\begin{lemma}\label{i} If $\lg\in\h K^{E_c}$ then the minor $N_c$ satisfies
$$N_c\oc\lg^*=\o{\lc_\l(\lg)}\ N_c$$
where
$$\lc_\l(\lg):=\lD_c(\lg c)$$
is a holomorphic character of $\h K^{E_c}.$
\end{lemma}
\begin{proof} Since $\lg\in\h K^{E_c}$ preserves $E_c$ the Peirce 2-projection $P_c$ satisfies $\lg P_c=P_c\lg P_c.$ Therefore 
$P_c\lg^*=P_c\lg^*P_c.$ Since $P_c\lg P_c\in\h K_{E_c}$ the semi-invariance of $\lD_c$ implies
$$N_c(\lg^*z)=\lD_c(P_c\lg^*z)=\lD_c(P_c\lg^*P_cz)=\lD_c(P_c\lg^*c)\ \lD_c(P_cz)=\o{\lD_c(\lg c)}\ N_c(z).$$
Here we use
$$\lD_c(P_c\lg^*c)=\lD_c(P_c\lg^*P_c c)=\lD_c((P_c\lg P_c)^*c)=\o{\lD_c(P_c\lg P_c c)}=\o{\lD_c(\lg c)}.$$
\end{proof}

As a consequence we obtain a {\bf holomorphic homogeneous line bundle} 
\be{27}\LL_\l:=\{[h,\lx]=[h\lg,\lc_\l(\lg)\lx]:\ h\in\h K,\ \lg\in\h K^{E_c},\ \lx\in\Cl\}\ee
over $M_\l,$ with projection $[h,\lx]\mapsto hE_c.$ By definition, the holomorphic sections $\ls:M_\l\to\LL_\l$ have the form
$$\ls_{hE_c}=[h,\t\ls(h)]$$
where the "homogeneous lift" $\t\ls:\h K\to\Cl$ is a holomorphic map satisfying
$$\t\ls(h\lg)=\lc_\l(\lg)\ \t\ls(h)$$
for all $h\in\h K,\ \lg\in\h K^{E_c}.$

Let $\ll\in\Nl_+^r$ be a partition written in the form \er{24}. The {\bf Peirce flag manifold} $M_{\l_1,\;,\l_t}$ consists of all chains of Peirce 2-spaces $U_1\ic U_2\ic \;\ic U_t$ such that $\x{rank}(U_s)=\l_s$ for all $s\le t.$ The group $\h K$ acts on $M_{\l_1,\;,\l_t}$ by holomorphic transformations $(h,(U_1,\;,U_t))\qi(hU_1,\;,hU_t)$ and the restriction to $K$ is already transitive. Thus
$$M_{\l_1,\;,\l_t}=\h K/\h K^{E_{c_1},\;,E_{c_t}}=K/K^{E_{c_1},\;,E_{c_t}}$$
where $\h K^{E_{c_1},\;,E_{c_t}}:=\{\lg\in\h K:\ \lg E_{c_s}=E_{c_s}\ \forall\ 1\le s\le t\}$ is the closed complex subgroup fixing the Peirce 2-flag $E_{c_1}\ic\;\ic E_{c_t}$ for a chain of tripotents $c_1<\;<c_t$ of rank $\l_1<\;<\l_t.$ We may choose 
$c_s:=e_{[\l_s]}.$ As a consequence of Lemma \ref{i} the conical function 
$$N^\ll:=\P_{s=1}^t N_{c_s}^{n_s-n_{s+1}}=N_{c_1}^{n_1-n_2}N_{c_2}^{n_2-n_3}\:N_{c_t}^{n_t}$$ 
satisfies
\be{40}N^\ll\oc\lg^*=\o{\lc^\ll(\lg)}\ N^\ll\ee
for all $\lg\in\h K^{E_{c_1},\;,E_{c_t}},$ where
$$\lc^\ll(\lg):=\P_{s=1}^t\lD_{c_s}(\lg c_s)^{n_s-n_{s+1}}$$
is a holomorphic character of $\h K^{E_{c_1},\;,E_{c_t}}.$ As a consequence we obtain a holomorphic homogeneous line bundle 
\be{28}\LL^\ll:=\{[h,\lx]=[h\lg,\lc^\ll(\lg)\lx]:\ h\in\h K,\ \lg\in\h K^{E_{c_1},\;,E_{c_t}},\ \lx\in\Cl\}\ee
over $M_{\l_1,\;,\l_t},$ with projection $[h,\lx]\mapsto(hE_{c_1},\;,hE_{c_t}).$ By definition, the holomorphic sections 
$\ls:M_{\l_1,\;,\l_t}\to\LL^\ll$ have the form
$$\ls_{hE_{c_1},\;,hE_{c_t}}=[h,\t\ls(h)]$$
where the homogeneous lift $\t\ls:\h K\to\Cl$ is a holomorphic map satisfying
$$\t\ls(h\lg)=\lc^\ll(\lg)\ \t\ls(h)$$
for all $h\in\h K$ and $\lg\in\h K^{E_{c_1},\;,E_{c_t}}.$ For $t=1$ we obtain the powers $\LL_\l^n$ of the line bundle \er{27}. 

We first describe the maximal fibre $J_0^\ll=\PL_E^\ll$ at $\lz=0.$

\begin{theorem}\label{j} Any $p\in\PL_E^\ll$ defines a holomorphic section $\ls^p:M_{\l_1,\;,\l_t}\to\LL^\ll,$ whose homogeneous lift $\t\ls^p:\h K\to\Cl$ is given by the Fischer-Fock inner product
\be{29}\t\ls^p(h)=(N^\ll\oc h^*|p)=(N^\ll|p\oc h).\ee 
This yields a $\h K$-equivariant isomorphism
$$\PL_E^\ll\xrightarrow[\ap]{\lT}\lG(M_{\l_1,\;,\l_t},\LL^\ll),\ p\qi\ls^p.$$
\end{theorem}
\begin{proof} To check the homogeneity condition, let $\lg\in\h K^{E_{c_1},\;,E_{c_t}}.$ Then \er{40} implies
$$\t\ls^p(h\lg)=(N^\ll\oc(h\lg)^*|p)=(N^\ll\oc\lg^*\oc h^*|p)$$
$$=(\o{\lc^\ll(\lg)}\ N^\ll\oc h^*|p)=\lc^\ll(\lg)\ (N^\ll\oc h^*|p)=\lc^\ll(\lg)\ \t\ls^p(h).$$
Since $\PL_E^\ll$ is the linear span of polynomials $N^\ll\oc h^*$ the map $\lT$ is injective. For any $h,h'\in\h K$ we have
$$\t\ls^{p\oc h}(h')=(N^\ll|(p\oc h)\oc h')=(N^\ll|(p\oc(hh'))=\t\ls^p(hh')=\w{(h^{-1}\.\ls^p)}(h'),$$
proving invariance
$$\ls^{p\oc h}=h^{-1}\.\ls^p$$
under $h\in\h K.$ Realizing $M_{\l_1,\;,\l_t}=\h K/\h K^{E_{c_1},\;,E_{c_t}}$ as a Lie theoretic flag manifold \cite{AU} the Borel-Weil-Bott theorem shows that $\lG(M_{\l_1,\;,\l_t},\LL^\ll)$ is $\h K$-irreducible. Thus the map $\lT$ is surjective and hence an isomorphism.
\end{proof}

As an example, consider the partitions
$$n^{(\l)}:=(n,\;,n,0,\;,0)$$
with $n>0$ repeated $\l$ times. Here $t=1$ and one obtains an isomorphism (cf. \cite{AU})
$$\PL_E^{n^{(\l)}}\ap\lG(M_\l,\LL_\l^n),$$
where $\LL_\l^n$ is the $n$-th tensor power of the line bundle $\LL_\l$ defined by \er{27}. If $\l=r$ and $E$ is unital with unit element $e,$ then $M_r=\{E\}$ is a singleton and $\PL_E^{n^{(r)}}$ is 1-dimensional, spanned by $N_e^n.$ If $\l$ is arbitrary and $n=1$ we have the "fundamental" partitions and 
$$\PL_E^{1^{(\l)}}\ap\lG(M_\l,\LL_\l).$$ 
In the simplest case $\l=1$ we obtain a kind of projective space $M_1$ consisting of all Peirce 2-spaces of rank 1, or equivalently, all lines spanned by tripotents of rank 1. The associated "tautological" line bundle
$$\TL=\U_{U\in M_1}U$$
has no holomorphic sections. On the other hand each $f\in E^*$ yields a holomorphic section $\ls^f:M\to\TL^*$ of the dual line bundle
$$\TL^*=\U_{U\in M_1}U^*$$
via
$$M_1\ni U\mapsto\ls^f_U:=f|_U\in\TL_U^*.$$
The calculation
$$\ls_{hU}^{f\oc h^{-1}}=(f\oc h^{-1})|_{hU}=(f|_U)\oc h^{-1}=\ls_U^f\oc h^{-1}=(h\.\ls^f)_{hU}$$
shows that $h\.\ls^f=\ls^{f\oc h^{-1}}.$ The homogeneous lift $\t\ls^f:\h K\to\Cl$ is given by
$$\t\ls^f(h):=f(he_1).$$
If $\lg\in\h K^{E_{e_1}}$ then $\lg e_1=\lc_1(\lg)e_1$ for the holomorphic character $\lc_1:\h K^{E_{e_1}}\to\Cl^\xx$ given by 
$\lc_1(\lg)=(\lg e_1|e_1)=N_{e_1}(\lg e_1).$ This is covered by the general formula \er{29}:

\begin{lemma} For each $f\in E^*=\PL_E^{1,0,\;,0}$ we have 
$$\t\ls^f(h)=(N_{e_1}\oc h^*|f)$$
\end{lemma}
\begin{proof} Putting $fz:=(z|v)$ the relation
$$(N_{e_1}\oc h^*)(z)=N_{e_1}(h^*z)=(h^*z|e_1)=(z|he_1)$$
implies
$$(N_{e_1}\oc h^*|f)=((\.|he_1)|(\.|v))=(he_1|v)=f(he_1).$$
\end{proof}

In order to describe the localization $J_\lz^\ll$ at non-zero points $\lz$ of rank $\l,$ where $\l_{s-1}\le\l<\l_s,$ we pass to certain submanifolds of the Peirce flag manifold. Consider the "incidence space"
$$M_{\l_s,\;,\l_t}^\lz:=\{(U_s,\;,U_t)\in M_{\l_s,\;,\l_t}:\ \lz\in U_s\}$$
and the "polar space"
$$M_{\l_s-\l,\;,\l_t-\l}^{\lz*}:=\{(W_s,\;,W_t)\in M_{\l_s-\l,\;,\l_t-\l}:\ W_t\b\lz^*=0\}.$$
For a tripotent $\lz=c$ this means that the flag is contained in $W=E^c.$ These compact submanifolds are closely related: If 
$(U_s,\;U_t)\in M_{\l_s,\;,\l_t}^\lz$ then the Peirce 0-spaces $W_k:=U_k^c$ form a flag 
$(W_s,\;,W_t)\in M_{\l_s-\l,\;,\l_t-\l}^{\lz*}.$ Conversely, if $(W_s,\;,W_t)\in M_{\l_s-\l,\;,\l_t-\l}^{\lz*}$ and we put 
$W_k:=W_{c_k}^2$ for a tripotent chain $c_s<\;<c_t$ of rank $\l_s-\l<\:<\l_t-\l$ in $W,$ then $c+c_s<\;<c+c_t$ is a tripotent chain of rank $\l_s<\;<\l_t$ in $E$ and $U_k:=E_{c+c_k}^2$ defines a flag $(U_s,\;U_t)\in M_{\l_s,\;,\l_t}^\lz.$ Moreover, if $c_k'\sim c_k$ are Peirce equivalent in $W$, then $c+c_k'\sim c+c_k$ are Peirce equivalent in $E.$  

The closed complex subgroup 
$$\h K^\lz:=\{\lk\in\h K: \lk\lz=\lz\}$$
acts transitively on the incidence space $M_{\l_s,\;,\l_t}^\lz$ and, similarly, the closed complex subgroup
$$\h K^{\lz*}:=\{\lk\in\h K: \lk^*\in\h K^\lz\}=\{\lk\in\h K: \lk^*\lz=\lz\}$$
acts transitively on the polar space $M_{\l_s-\l,\;,\l_t-\l}^{\lz*}$ since $h(z\b\lz^*)h^{-1}=(hz)\b(h^{-*}\lz)^*$ for all 
$h\in\h K.$ Therefore
$$M_{\l_s-\l,\;,\l_t-\l}^{\lz*}=\h K^{\lz*}/(\h K^{\lz*})^{W_{c_s},\;,W_{c_t}}$$
where
$$(\h K^{\lz*})^{W_{c_s},\;,W_{c_t}}=\{\lg\in\h K^{\lz*}:\ \lg W_{c_j}=W_{c_j}\ \forall\ s\le j\le t\}$$
and $W_{c_s}\ic\;\ic W_{c_t}\ic W$ is the Peirce 2-flag for a given chain of tripotents $c_s<\;<c_t$ in $W$ of type 
$\l_s-\l<\;<\l_t-\l.$ We may choose $c_j:=e_{\l+1}\;+e_{\l_j}.$ The restriction homomorphism
$$\h K^{\lz*}\to\h K_W,\ \lk\qi\lk|_W$$ 
defines a biholomorphic map
\be{30}M_{\l_s-\l,\;,\l_t-\l}^{\lz*}=\h K^{\lz*}/(\h K^{\lz*})^{W_{c_s}\;,W_{c_t}}\xrightarrow[\ap]{\lr}
\h K_W/\h K_W^{W_{c_s}\;,W_{c_t}}=M_{\l_s-\l,\;,\l_t-\l}^W\ee
where 
$$\h K_W^{W_{c_s}\;,W_{c_t}}=\{\lg\in\h K_W:\ \lg W_{c_j}=W_{c_j}\ \forall\ s\le j\le t\}$$
and $M_{\l_s-\l,\;,\l_t-\l}^W$ denotes the Peirce flag manifold relative to $W.$ Now consider the truncated partition $\ll^*$ of length $r-\l.$ Then the conical function relative to $W$ is
$$N_W^{\ll^*}=\P_{j=s}^t\t N_{c_j}^{n_j-n_{j+1}}$$
and there is a holomorphic character
$$\lc_W^{\ll^*}(\lg):=\P_{j=s}^t\t N_{c_j}(\lg c_j)^{n_j-n_{j+1}}$$
on $\h K_W^{W_{c_s},\;,W_{c_t}}$ such that
$$N_W^{\ll^*}\oc\lg^*=\o{\lc_W^{\ll^*}(\lg)}\ N_W^{\ll^*}$$
for all $\lg\in\h K_W^{W_{c_s},\;,W_{c_t}}.$ Let $\LL_W^{\ll^*}$ denote the associated holomorphic line bundle on 
$M_{\l_s-\l,\;,\l_t-\l}^W$ defined as in \er{28}. A holomorphic section $\ls:M_{\l_s-\l,\;,\l_t-\l}^W\to\LL_W^{\ll^*}$ is characterized by its homogeneous lift $\t\ls:\h K_W\to\Cl$ satisfying
$$\t\ls(\lk\lg)=\lc_W^{\ll^*}(\lg)\ \t\ls(\lk)$$
for all $\lk\in\h K_W$ and $\lg\in\h K_W^{W_{c_s}\;,W_{c_t}}.$ Via the isomorphism \er{30} we obtain a holomorphic pull-back line bundle $\LL_\lz^{\ll^*}=\lr^*\LL_W^{\ll^*}$ such that holomorphic sections 
$\ls:M_{\l_s-\l,\;,\l_t-\l}^{\lz*}\to\LL_\lz^{\ll^*}$ are characterized by the homogeneous lift $\t\ls:\h K^{\lz*}\to\Cl$ satisfying
$$\t\ls(\lk\lg)=\lc_W^{\ll^*}(\lg|_W)\ \t\ls(\lk)=\P_{j=s}^t\t N_{c_j}(\lg c_j)^{n_j-n_{j+1}}\ \t\ls(\lk)$$
for all $\lk\in\h K^{\lz*}$ and $\lg\in(\h K^{\lz*})^{W_{c_s}\;,W_{c_t}}.$ All of this holds in particular when $\lz=c$ is a tripotent of rank $\l.$

\begin{theorem} Write $\ll\in\Nl_+^r$ in the form \er{24}. Let $c$ be a tripotent of rank $\l$ with Peirce 0-space $W,$ such that 
$\l_{s-1}\le\l<\l_s.$ Consider the truncated partition
$$\ll^*=(\ll_{\l+1},\;,\ll_r)=(n_s^{(\l_s-\l)},n_{s+1}^{(\l_{s+1}-\l_s)},\;,n_t^{(\l_t-\l_{t-1})}).$$
Then there is an isomorphism
$$J_c^\ll\ap\PL_W^{\ll^*}\xrightarrow[\ap]{\lT_c}\lG(M_{\l_s,\;,\l_t}^{c*},\LL_c^{\ll^*})$$
which is defined as follows: Any $\lf\in\PL_W^{\ll^*}$ defines a holomorphic section 
$\ls^\lf:M_{\l_s,\;,\l_t}^{c*}\to\LL_c^{\ll^*},$ whose homogeneous lift $\t\ls^\lf:\h K^{c*}\to\Cl$ is given by the Fischer-Fock inner product
$$\t\ls^\lf(\lk)=(N_W^{\ll^*}|\lf\oc\lk|_W)$$
for all $\lk\in\h K^{c*}.$ 
\end{theorem}
\begin{proof} Applying Theorem \ref{j} to $W$ we have
$$\PL_W^{\ll^*}\xrightarrow[\ap]{\lT_W}\lG(M_{\l_s-\l,\;,\l_t-\l}^W,\LL_W^{\ll^*})$$
via the map $\lf\mapsto\ls_W^\lf$ with homogeneous lift
$$\t\ls_W^\lf(\lk)=(N_W^{\ll^*}|\lf\oc\lk)$$
for all $\lk\in\h K_W.$ Passing to $\lG(M_{\l_s,\;,\l_t}^{c*},\LL_c^{\ll^*})$ via the isomorphism \er{30}, the assertion follows.
\end{proof}

We finally describe a similar isomorphism for non-tripotent points $\lz\in\c E_\l.$ Write $\lz=hc$ for some $h\in\h K.$ Then
$h^{-*}\h K^{c*}h^*=\h K^{\lz*}$ and 
$$h^{-*}(\h K^{c*})^{W_{c_s},\;,W_{c_t}}h^*=(\h K^{\lz*})^{h^{-*}W_{c_s},\;,h^{-*}W_{c_t}}$$ 
for the Peirce 2-flag $h^{-*}W_{c_s}\ic\;\ic h^{-*}W_{c_t}\ic h^{-*}W\ic E.$ Thus we obtain a commuting diagram
$$\xymatrix{\h K^{c*}/(\h K^{c*})^{W_{c_s}\;,W_{c_t}}\ar[r]\ar[d]_{h^{-*}\.h^*}&M_{\l_s-\l,\;,\l_t-\l}^{c*}\ar[d]^{h^{-*}}
\\\h K^{\lz*}/(\h K^{\lz*})^{h^{-*}W_{c_s}\;,h^{-*}W_{c_t}}\ar[r]&M_{\l_s-\l,\;,\l_t-\l}^{\lz*}}.$$
Now the isomorphism
$$J_\lz^\ll\xrightarrow[\ap]{\lT_\lz}\lG(\h K^{\lz*}/(\h K^{\lz*})^{h^{-*}W_{c_s}\;,h^{-*}W_{c_t}},\LL_\lz^{\ll^*}).$$
is defined via the commuting diagram
$$\xymatrix{J_c^\ll\ar[rrr]_\ap^{\lT_c}&&&\lG(\h K^{c*}/(\h K^{c*})^{W_{c_s}\;,W_{c_t}},\LL_c^{\ll^*})\\
J_\lz^\ll\ar[rrr]_\ap^{\lT_\lz}\ar[u]_{\ap}^{\oc h}&&&\lG(\h K^{\lz*}/(\h K^{\lz*})^{h^{-*}W_{c_s}\;,h^{-*}W_{c_t}},\LL_\lz^{\ll^*})\ar[u]_{h^*\.h^{-*}}^\ap}.$$
This demonstrates that the bundle $J_\lz^\ll$ depends in an anti-holomorphic way on $\lz.$

{\bf Data Availability Statement}: The author confirms that the data supporting the findings of this study are available within the article.

\end{document}